\documentclass[12pt]{article}%preprint,review,
 \usepackage{graphics}
\usepackage{graphicx}
\usepackage{amssymb}
\usepackage{url}
\usepackage{amsthm}

\theoremstyle{plain}
\newtheorem{theorem}{Theorem}
\newtheorem{lemma}[theorem]{Lemma}

\theoremstyle{definition}

\theoremstyle{remark}

\begin{document}

\title{Irreducible triangulations of the M\"{o}bius band}

\author{Mar\'{\i}a-Jos\'e Ch\'avez\\
        Departamento de Matem\'atica Aplicada I,\\
       Universidad de Sevilla, Spain\\
       mjchavez@us.es\\
\and Serge Lawrencenko\\
      Faculty of Control and Design\\
    Russian State University of Tourism and Service\\
    Lyubertsy, Moscow Region 140000, Russia\\
    lawrencenko@hotmail.com\\
 \and Antonio Quintero\\
      Departamento de Geometr\'{\i}a y Topolog\'{\i}a\\
      Universidad de Sevilla,  Spain\\
      quintero@us.es\\
 \and
    Mar\'{\i}a Trinidad Villar\\
    Departamento de Geometr\'{\i}a y Topolog\'{\i}a\\
  Universidad de Sevilla,  Spain\\
    villar@us.es\\
}

\maketitle

This work has been partially supported by PAI FQM-164; PAI FQM-189;
MTM 2010-20445.

\begin{abstract}
A complete list of irreducible triangulations is identified on the M\"{o}bius band.

\end{abstract}



\section{Introduction}\label{}
Let $S\in\{S_g, N_k\}$ be the closed orientable surface $S_g$ of
genus $g$ or the closed non-orientable surface $N_k$  of
non-orientable  genus $k$. In particular, $S_0$ is the sphere and
$N_1$ is the projective plane. Let $D$ be an open disk in $S$, so
that the boundary $\partial(S-D)$ (=$\partial D$) is homeomorphic to
a circle. In particular, $S_0-D$ is the disk and $N_1-D$ is the
M\"{o}bius band. We use the notation $\Sigma$ whenever we assume the
general case: $\Sigma\in\{S, S-D\}$ .

If a graph $G$ is 2-cell embedded in $\Sigma$, the components of
$\Sigma -G$  are called \textit{faces}. A \textit{triangulation} of
$\Sigma$ with a simple graph $G$ (without loops or multiple edges)
is a 2-cell embedding $T : G \rightarrow \Sigma$  in which each face
is bounded by a 3-cycle (that is, a cycle of length 3) of $G$  and
any two faces are either disjoint, share a single vertex, or share a
single edge. We denote by $V=V(T)$, $E=E(T)$, and $F=F(T)$  the sets
of vertices, edges, and faces of $T$, respectively. The cardinality
$|V(T)|$ is called the \textit{order} of $T$. By $G(T)$  we denote
the graph $(V(T), E(T))$  of triangulation $T$. Two triangulations
$T_1$ and $T_2$ are called \textit{isomorphic} if there is a
bijection, called an \textit{isomorphism}, $\varphi : V(T_1)
\rightarrow V(T_2)$, such that $uvw\in F(T_1)$  if and only if
$\varphi(u)\varphi(v)\varphi(w)\in F(T_2)$. Throughout this paper we
distinguish triangulations only up to isomorphism. For $\Sigma = S-D
$, let $\partial T$  ($=\partial D$) denote the boundary cycle of
$T$. The vertices and edges of $\partial T$ are called
\textit{boundary vertices} and \textit{boundary edges} of $T$.

A triangulation is called \textit{irreducible} if no edge can be shrunk
without producing multiple edges or changing the topological type of
the surface. The term ``irreducible triangulation'' is more accurately
introduced in Section 2. The irreducible triangulations of $\Sigma$ form a
basis for the family of all triangulations of $\Sigma$, in the sense that any
triangulation of $\Sigma$  can be obtained from a member of the basis by
repeatedly applying the \textit{splitting} operation (introduced in Section 2)
a finite number of times. Barnette and Edelson \cite{BE1989} and
independently Negami \cite{N2001} have proved that for every closed
surface $S$ the basis of irreducible triangulations is finite. At
present such bases are known for seven closed surfaces: the sphere
(Steinitz and Rademacher \cite{SR1976}), projective plane (Barnette
\cite{B1982}), torus (Lawrencenko \cite{L1987}), Klein bottle (25
Lawrencenko and Negami's \cite{LN1997} triangulations plus 4 more
irreducible triangulations found later by Sulanke \cite{S2006}) as
well as $S_2$,  $N_3$, and  $N_4$  (Sulanke \cite{S2006-2,
S2006-4}). Boulch, Colin de Verdi\`{e}re, and Nakamoto
\cite{BCN2011} have established upper bounds on the order of an
irreducible triangulation of $S - D$. In this paper we obtain a
complete list of irreducible triangulations of $N_1 - D$.

\section{Preliminaries}

Let $T$ be a triangulation of $\Sigma$. An unordered pair of
distinct adjacent edges $vu$  and $vw$ of $T$ is called a
\textit{corner} of $T$  at vertex $v$, denoted by $\langle u, v, w
\rangle$. The \textit{splitting} of a corner $\langle u, v, w
\rangle$, denoted by sp$\langle u, v, w \rangle$, is the operation
which consists of cutting $T$ open along the edges $vu$ and $vw$ and
then closing the resulting hole with two new triangular faces,
$v'v''u$ and $v'v''w$, where $v'$ and $v''$ denote the two images of
$v$ appearing as a result of cutting. Under this operation, vertex
$v$ is extended to the edge $v'v''$ and the two faces having this
edge in common are inserted into the triangulation. Especially in
the case $\{\Sigma= S -D \wedge uv\in E(T) \wedge v\in V(\partial
T)\}$, the operation sp$\langle u, v] $ of \textit{splitting a
truncated corner} $\langle u, v] $ produces a single triangular face
$uv'v''$, where $v'v''\in E(\partial ({\rm{sp}} \langle u, v](T)))$.

Under the
inverse operation, \textit{shrinking} the edge $v'v''$, denoted by
sh$\rangle v'v''\langle $, this edge collapses to a single vertex
$v$, the faces $v'v''u$  and $v'v''w$ collapse to the edges $vu$ and
$vw$, respectively. Therefore sh$\rangle v'v''\langle \circ$
sp$\langle u, v, w \rangle (T)=T$. It should be noticed that in the
case $\{\Sigma= S -D \wedge v'v''\in E(\partial T)\}$, there is only
one face incident with $v'v''$, and only that single face collapses
to an edge under sh$\rangle v'v''\langle $. Clearly, the
 operation of splitting doesn't change the topological type of $\Sigma$.
 We demand that the shrinking operation must preserve the topological type of $\Sigma$  as well; moreover,
 multiple edges must not be created in a triangulation. A 3-cycle  of $T$ is called \textit{nonfacial} if it
doesn't bound a face of $T$. In the case in which an edge $e\in
E(T)$ occurs in some nonfacial 3-cycle, if we still insist on
shrinking $e$, multiple edges would be produced, which would expel
sh$\rangle e\langle (T)$  from the class of triangulations. An edge
$e$ is called \textit{shrinkable} or a \textit{cable} if sh$\rangle
e\langle (T)$ is still a triangulation of $\Sigma$; otherwise the
edge is called \textit{unshrinkable} or a \textit{rod}. The subgraph
of $G(T)$  made up of all cables is called the
\textit{cable-subgraph }of $G(T)$.

The only impediment to edge shrinkability in a triangulation $T$ of a
closed surface  $S$ is identified in \cite{B1982, BE1989, L1987}: an edge $e\in E(T)$
is a rod if and only if $e$ satisfies the following condition:

(2.1) $e$   is in a nonfacial 3-cycle of $G(T)$.

The impediments to edge shrinkability in a triangulation $T$ of a punctured surface $S- D$ are
identified in \cite{BCN2011}: an edge $e\in E(T)$ is a rod if and only if $e$  satisfies either
condition (2.1) or the following condition:

(2.2)  $e$ is a \textit{chord } of $D$---that is, the end vertices of
$e$  are in $V(\partial D)$ but $e\notin E(\partial D)$.

A triangulation is said to be \textit{irreducible} if it is free of cables or in other words, each edge is a rod.
For instance, a single triangle is the only irreducible triangulation of the disk $S_0-D$  although its edges don't meet either of conditions (2.1) and (2.2). Thus, we have yet one more impediment to edge shrinkability:

(2.3)  $e$  is a boundary edge in the case the boundary cycle is a 3-cycle.

Although condition (2.3) is a specific case of condition (2.1) (unless $S=S_0$) and is not
explicitly stated in \cite{BCN2011}, it deserves especial mention.

\section{The structure of irreducible punctured triangulations}

In the remainder of this paper we assume that $S \neq S_0$. Let $T$
be an irreducible triangulation of $S-D$. Let us restore the disk
$D$ in $T$, add a  vertex $p$ in $D$ and join $p$ to the
vertices in $\partial D$. We thus obtain a triangulation, $T^*$, of
the closed surface $S$. In this setting we call $D$ the
\textit{patch}, call $p$ the \textit{central vertex of the patch},
and say that $T$ is obtained from the corresponding triangulation
$T^*$ of $S$ by the \textit{patch removal}. Notice that  $T^*$ may
turn out to be an irreducible triangulation of $S$, but not
necessarily.

A vertex of a triangulation $R$ of  $S$ is called a \textit{pylonic vertex} if that
vertex is incident with all cables of $R$. A triangulation that has at
least one cable and at least one pylonic vertex is called a \textit{pylonic
triangulation}. It should be noticed that there exist
triangulations of the torus with exactly one
cable, and thereby with two different pylonic vertices; however, if
a pylonic triangulation $R$ has at least two cables, $R$ has a unique
pylonic vertex.

\begin{lemma}\label{lem1} If $T^*$ has at least two cables, then the central vertex
$p$ of  the patch is the only pylonic vertex of  $T^*$.
\end{lemma}

\begin{proof} Using the assumption that $T$ is irreducible and the fact
that each cable of  $T^*$ fails to satisfy condition (2.1), it can be
easily seen that in   the case  $T^*$  is not irreducible, all cables of  $T^*$ have
to be entirely in $D \cup \partial D$ and, moreover, there is no cable that is entirely
in $ \partial D$. In particular, we observe that any chord of $D$  is a rod in  $T$
because it meets condition (2.2), and is also a rod in  $T^*$  because it
meets condition (2.1).
\end{proof}

\section{Irreducible triangulations of the M\"{o}bius band}

Barnette's theorem \cite{B1982} states that there exist two
irreducible triangulations of $N_1$; those are presented in Figure
1: $P_1$ and $P_2$. (For each hexagon identify each antipodal pair
of vertices in the boundary to obtain an actual triangulation of
$N_1$.) By repeatedly applying the splitting operation to  $P_1$
and $P_2$, we can generate all triangulations of $N_1$. Sulanke
\cite{S2005} has generated by computer all triangulations of $N_1$
with up to 19 vertices; in particular, among them there are 20
triangulations with up to 8 vertices. Independently, the authors of
the present paper have identified the same list of 20 triangulations
by hand (Figure 1), using the automorphisms of  $P_1$ and $P_2$. An
\textit{automorphism} of a triangulation  $P$ is an isomorphism of
$P$ with itself. The set of all automorphisms of  $P$ forms a group,
called the \textit{automorphism group} of  $P$ (denoted  ${\rm{Aut}}(P)$).

\begin{figure}[!h]
  \begin{center}
     \includegraphics[width = 12cm]{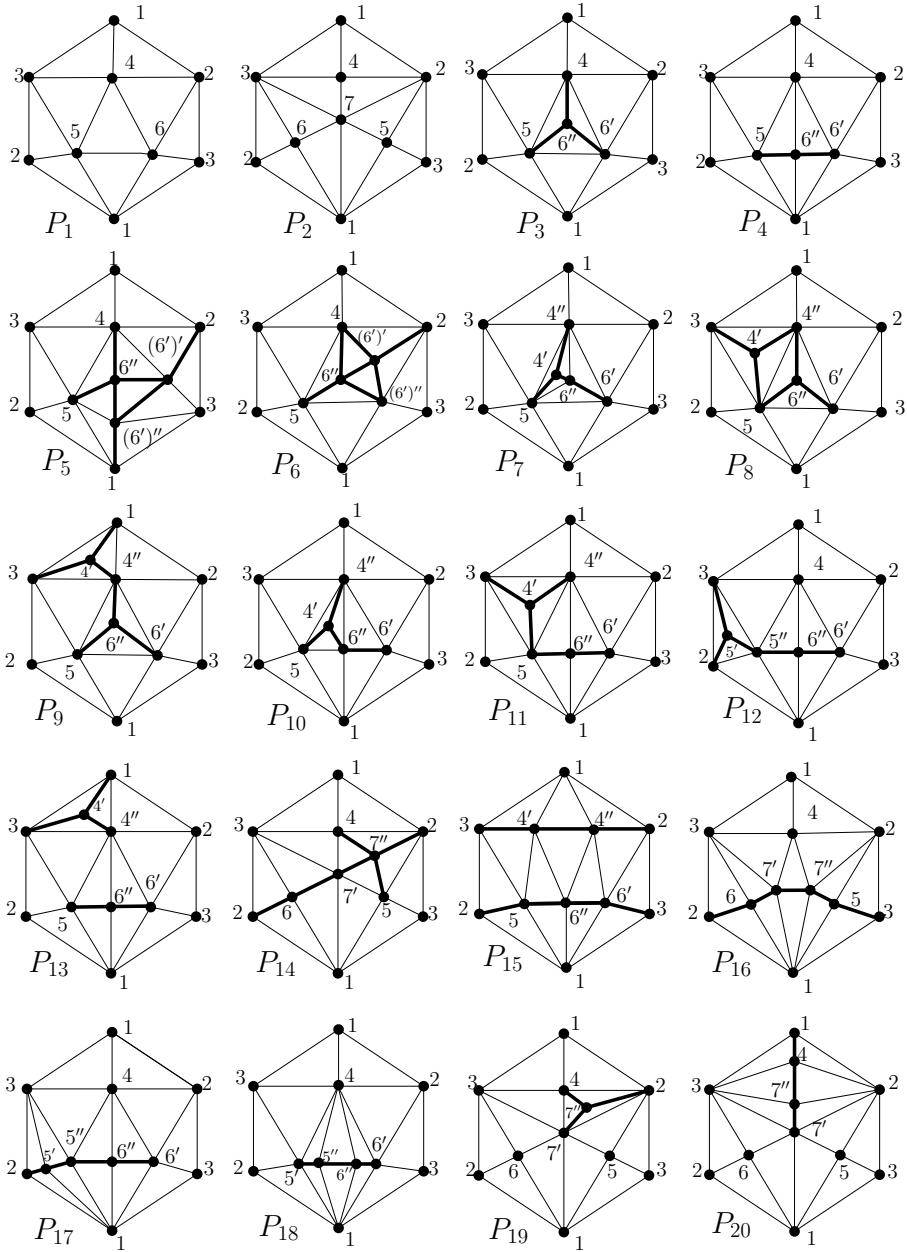}
  \end{center}
  \caption{\label{fig1} All projective plane triangulations with up
  to 8 vertices.}
\end{figure}

\begin{lemma}\label{lem2}{\sc{(\cite{S2005})}}. There are precisely one (up to isomorphism)
triangulation of  $N_1$ with 6 vertices, three with 7 vertices, and
sixteen with 8 vertices. They are shown in Figure \ref{fig1}, in
which the bold edges indicate the cable-subgraphs of the
triangulations.
\begin{flushright}$\Box$
\end{flushright}
\begin{figure}[!h]
  \begin{center}
       \includegraphics[width = 8cm]{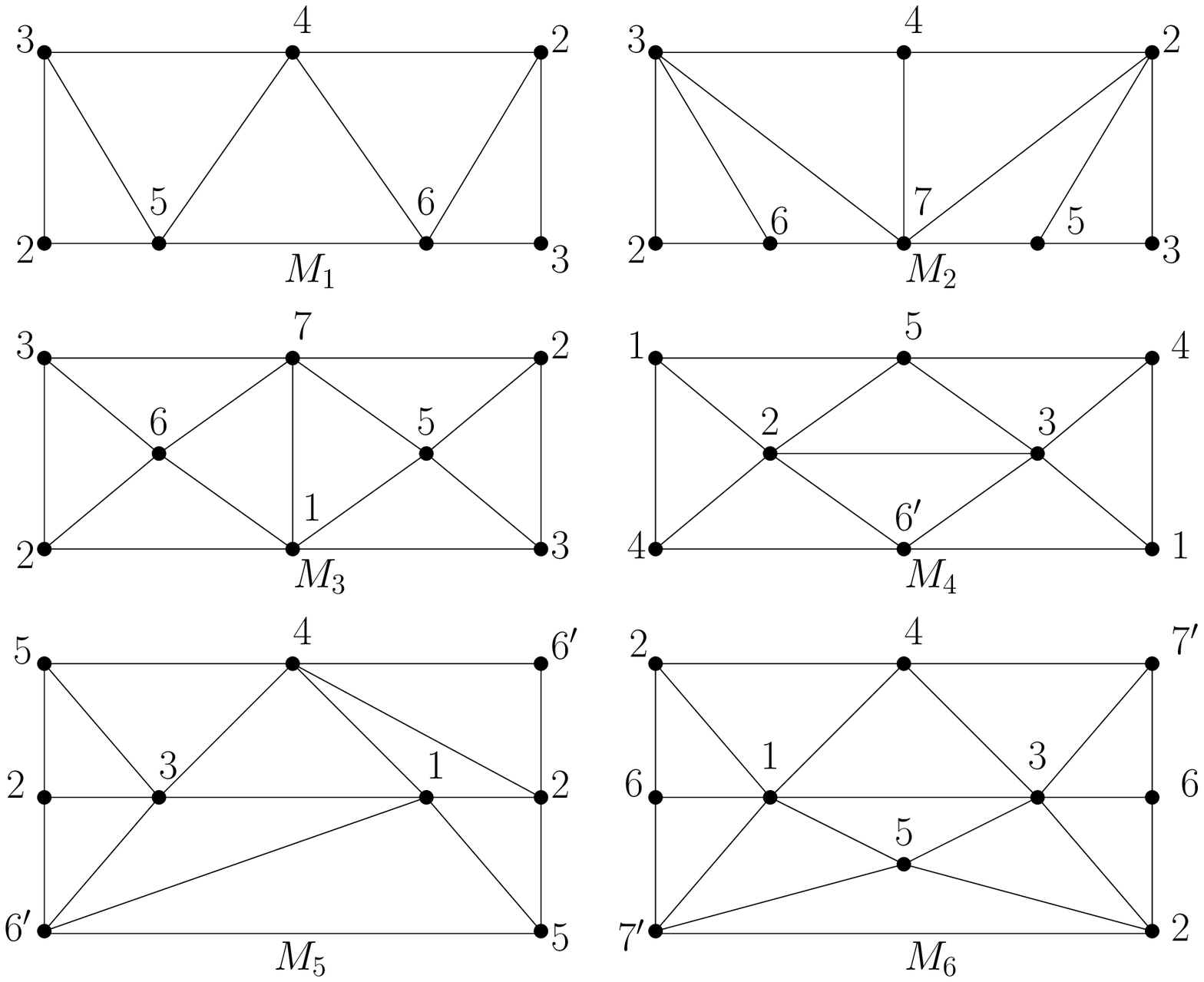}
  \end{center}
  \caption{\label{fig2} Irreducible triangulations of the M\"{o}bius band.}
\end{figure}
\end{lemma}

\begin{theorem} There are precisely six non-isomorphic irreducible
triangulations of the M\"{o}bius band, namely $M_1$  to $M_6$, shown
in Figure 2 in which the left and right sides of each rectangle are
identified with opposite orientation to obtain an actual
triangulation of the M\"{o}bius band.
\end{theorem}
\begin{proof} Observe that in Figure 1 only the following three non-irreducible members have
a pylonic vertex:  $P_3$ and  $P_4$ with pylonic vertex $6''$, and
$P_{19}$ with pylonic vertex $7''$. It can be easily proved that if
a triangulation of  $N_1$ has at least two cables but has no pylonic
vertex, then no pylonic vertex can be created under further
splitting of the triangulation. On the other hand, it can be easily
seen that any one splitting applied to the pylonic triangulations
$P_3$, $P_4$, or $P_{19}$ destroys their pylonicity. Therefore, by
Lemma \ref{lem1}, each irreducible triangulation of  $N_1 -D$ is
obtainable either by removing a vertex from an irreducible
triangulation in $\{ P_1 , P_2 \}$, or by removing the pylonic
vertex from a pylonic triangulation in $\{ P_3 , P_4 , P_{19} \}$.
It is known \cite{CKL2000, CL1998, L1992} that ${\rm{Aut}}(P_1)$ acts
transitively on the vertex set $V(P_1)$, while under the action of
${\rm{Aut}}(P_2)$ the set  $V(P_2)$ breaks into two orbits as follows:
${\rm{orbit}}_1= \{  1,2,3,7 \}$, ${\rm{orbit}}_2=\{  4,5,6 \}$. Therefore, all
irreducible triangulations of $N_1-D$  are covered by the followings:
$M_1=P_1$ minus vertex $1$ (subtructed with the incident edges and faces),
$M_2 =P_2$ minus vertex $1$, $M_3 =P_2$ minus vertex $4$,
$M_4 =P_4$ minus vertex $6''$, $M_5=P_3$ minus vertex $6''$,
$M_6 =P_{19}$  minus vertex $7''$. To see that these triangulations
are pairwise non-isomorphic, observe that they have different vertex
degree sequences except for the pair $\{ M_3 ,M_4 \}$; however, all
boundary vertices have degree 5 in  $M_3$ but not all in $M_4$.
\end{proof}




% \end{thebibliography}

\end{document}